\documentclass[a4paper,12pt,english]{article}

\usepackage{amsfonts,bm,amssymb,euscript,array,babel,cite}
\usepackage{relsize}
\usepackage{tikz}
\usepackage{url}
\usepackage{amsmath,amsthm}
\usepackage[]{graphicx}
\usepackage[T1]{fontenc}

\makeatletter

\newcommand{\bbm}{\left(\begin{matrix}}
\newcommand{\ebm}{\end{matrix}\right)}
\newcommand{\beq}{\begin{eqnarray}}
\newcommand{\eeq}{\end{eqnarray}}

\makeatother

\newtheorem{prop}{Proposition}
\newtheorem{theorem}[prop]{Theorem}
\newtheorem{lemma}[prop]{Lemma}

\newtheorem{cor}[prop]{Corollary}
\theoremstyle{definition}
\newtheorem{rem}[prop]{Remark}

\newcommand{\be}{\begin{equation}}
\newcommand{\ee}{\end{equation}}

\newcommand{\beqa}{\begin{eqnarray}}
\newcommand{\eeqa}{\end{eqnarray}} 
 \def \bea{\begin{eqnarray}} \def\eea{\end{eqnarray}}

\newcommand{\barr}{\begin{array}}
\newcommand{\earr}{\end{array}}

 \def\one{\mbox{1 \kern-.59em {\rm l}}}

\def\bit{\begin{itemize}} \def\eit{\end{itemize}}

\def\({\left(} \def\){\right)}

 \allowdisplaybreaks[3]

\begin{document}




\renewcommand{\title}[1]{\vspace{10mm}\noindent{\Large{\bf

#1}}\vspace{8mm}} \newcommand{\authors}[1]{\noindent{\large

#1}\vspace{5mm}} \newcommand{\address}[1]{{\itshape #1\vspace{2mm}}}

\newcommand{\new}[1]{{\color{red}#1}}

\begin{center}
\vskip -9mm
\title{{\Large Transverse generalized metrics and 2d sigma models}} 
\vskip -3mm
\authors{Pavol \v Severa$~^{\sharp}$
 and Thomas Strobl$~^{\star}$}
\vskip -1mm
\address{
$^{\sharp}$Section of Mathematics, Universit\'e de Gen\`eve, Geneva, Switzerland\\
$^{\star}$Institut Camille Jordan, Universit\'e Claude Bernard Lyon 1, Villeurbanne, France}
\vskip 1mm
\textbf{Abstract}
\vskip 2mm
\begin{minipage}{15cm}%
We reformulate the compatibility condition between a generalized metric and a small (non-maximal rank) Dirac structure in an exact Courant algebroid found in the context of the gauging of strings and formulated by means of two connections in purely Dirac-geometric terms. The resulting notion, a transverse generalized metric, is also what is needed for the dynamics on the reduced phase space of a string theory.
\end{minipage}
\end{center}
\noindent {\bf 1.} Let $E$ be an exact Courant algebroid over $M$, characterized by the class $[H] \in H^3_{dR}(M)$ \cite{letters,3form} and let $V \subset E$ be a \emph{generalized metric}, i.e.\ a positive definite, rank $n=\dim M$ subbundle of $E$.\footnote{See also, e.g., the first part of Section 3 in \cite{DSM} for a concise review of the required notions or Sections 2 and 3 of \cite{Cortes} for a review with further details.}
These data are equivalent to the choice of a Riemannian metric $g$ and a representative closed 3-form $H$ on $M$ (since there is a unique splitting of $\rho\colon E \to TM$
such that $V$ can be written as the graph of a symmetric 2-tensor). They are also the data  needed on the target space for the definition of a standard sigma model with Wess-Zumino term. Choose a \emph{small Dirac structure}, i.e.\ an involutive, isotropic $C^\infty(M)$-submodule ${\cal D}$ of $\Gamma(E)$. In this note we only consider regular $\cal D$'s, i.e.\ those of the form ${\cal D}=\Gamma(D)$ for some sub-vector bundle $D \subset E$.

We call a rank $n$ subbundle $W \subset E$ a pre-$D$-transverse generalized metric if $D \subset W \subset D^\perp$ and $\langle w,w\rangle>0$ for every $w \in W$ with $w \not \in D$. This becomes   a $D$-transverse generalized metric, or simply  a \emph{transverse generalized metric}, if in addition the invariance property 
\be [\Gamma(D),\Gamma(W)] \subset \Gamma(W) \,  \label{inv}
\ee 
holds true. 

\vspace{2mm}
\noindent {\bf 2.} If $D$ is such that $\rho|_D \colon D\to TM$ is injective (in which case we call $D$ \emph{projectable}), then a $D$-transverse generalized metric is equivalent to a Riemannian metric  and a closed 3-form, both on the space of leaves of the resulting foliation  $F:=\rho(D)\subset TM$. In more detail, we have:
\begin{prop} Suppose that the leaves of the foliation $F=\rho(D)$ generated by a projectable small Dirac structure $D$ are the fibers of a surjective submersion $\pi\colon M\to Q$. If $W\subset E$ is a $D$-transverse generalized metric, then there is  a unique splitting $E\cong(T\oplus T^*)M$ such that the resulting 3-form is of the form $\pi^*H_Q$ and $W$ is the graph of $\pi^*g_Q$, where, respectively, $H_Q$ and  $g_Q$ are a closed 3-form and a Riemannian metric  on $Q$.
\end{prop}
\begin{proof}
There is a unique splitting identifying $E$ with $(T\oplus T^*)M$ such that $W$ is the graph of a (degenerate) symmetric bilinear form $h$ on $TM$. Using this splitting, one has $D= F= \ker h$.
 
The condition $[\Gamma(D),\Gamma(W)]\subset\Gamma(W)$ means
$$\bigl[(X,0),(u,h(u,\cdot))\bigr]=\bigl([X,u], (\mathcal L_X h)(u,\cdot)+h([X,u],\cdot)+H(X,u,\cdot)\bigr)\in \Gamma(W)$$
for every $X\in\Gamma(F)$ and every vector field $u$, and thus
$$\mathcal L_X h=0,\qquad \iota_X H=0.$$
Together with $\iota_X h=0$, $\mathrm{d}H=0$, and the 
semi-positivity requirement on $W$, these two equations imply that $h$ and $H$ are the pullback of a 
Riemannian metric $g_Q$ and a closed 3-form $H_Q$, respectively.
\end{proof}

\noindent {\bf 3.} 
If $(M,g)$ is a Riemannian manifold and $\pi\colon M\to Q$ a submersion, then $\pi$ is called a Riemannian submersion iff $g$ descends to a Riemannian metric on $Q$. This means the following: for any $m\in M$ we decompose $T_mM$ orthogonally to $T_m^\parallel M\oplus T_m^\perp M$, where $T_m^\parallel M:=\ker d_m\pi$ is the subspace tangent to the fiber of $\pi$. Define $h$ as the unique symmetric bilinear form on $T_mM$ which agrees with $g$  on $T_m^\perp M$ and which gives zero when one of the vectors is in $T_m^\parallel$. Then $\pi$ is a Riemannian submersion iff $h=\pi^*g_Q$ for some Riemannian metric $g_Q$ on $Q$.

 We reformulate this in a language adapted to our purposes: Take $V$ and $W$ to be the graphs of $g$ and $h$ inside the standard Courant algebroid $E_0 := (T\oplus T^*)M$, respectively, and let $D:=\{(u,0) \in TM \oplus T^*M| \pi_* u = 0\}$. One verifies easily that then $W= D \oplus (D^\perp \cap V)$. As we saw above, the relation $h = \pi^*g_Q$ for some $g_Q$ is equivalent to $W$ being a $D$-transverse generalized metric. 

If one, more generally, has a foliation on a Riemannian manifold $(M,g)$ where there is no good quotient manifold $Q$ but where $g$ descends to the local spaces of leaves (i.e.\ to a transverse Riemannian metric), one talks of a (regular) Riemannian foliation \cite{Molino}. 
This motivates the following definition: 

A triple $(E,V,D)$---of an exact Courant algebroid $E$, a generalized metric $V\subset E$, and a small Dirac structure $D\subset E$---is a \emph{Dirac-Riemannian foliation} iff $V_D := D \oplus (D^\perp \cap V)$ is a $D$-transverse generalized metric.

\vspace{2mm}

\noindent {\bf 4.} In their simplest setting, sigma models are variational problems defined on the space of maps $X\colon \Sigma \to M$ from a $d$-dimensional, Lorentzian signature pseudo-Riemannian manifold $(\Sigma,\gamma)$ to a Riemannian manifold $(M,g)$. In the traditional setting of gauging such a theory, one needs to be given a group $G$ acting isometrically on $(M,g)$. The procedure is devised in such a way that when the group acts freely and properly on $M$, then the gauged sigma model is equivalent to the ungauged one with target manifold $M/G$. 

The gauging is achieved by coupling to $\mathrm{Lie}(G)$-valued 1-forms $A$ on $\Sigma$  in such a way that two maps $X_1$ and $X_2$ which differ from one another only by the application of a $\Sigma$-dependent group element are related by a symmetry of the action functional depending on $X$ and $A$. 
If the original model is twisted by a closed $d+1$-form $H$, called a Wess-Zumino term, then gauging in this way requires that $H$ has an equivariantly closed extension  \cite{gaugeability,Figueroa}. 

While every isometric $G$-action equips $(M,g)$ with an---in general only singular---Riemannian foliation, where the leaves are given by the $G$-orbits, not every such a foliation results from a $G$-action. According to \cite{without_symmetries,RMP}\footnote{See also Section 2 
of \cite{AHP} for a possible definition of such a generalized notion of gauging.}, it is not necessary to restrict gauging to eventual isometries of $g$, it is sufficient that  $(M,g)$ defines (a somewhat controlled form of) a singular Riemannian foliation. While   the general theory of such gaugings is not worked out yet in the presence of a Wess-Zumino twist for arbitrary dimensions $d$, it was done so in \cite{AHP} for $d=2$.  Two-dimensional sigma models are somewhat particular since the Hodge dual of a 1-form $A$ is again a 1-form, yielding additional options for the gauging, and they are in general intimately related to Dirac geometry, see, e.g., \cite{letters,AS}.

\vspace{2mm}

\noindent {\bf 5.} Let $V\subset E$ be a generalized metric and $D\subset E$ a small Dirac structure. Let us use the (unique) splitting $E\cong(T\oplus T^*)M$ that turns $V$ into the graph of a Riemann metric. The resulting inclusion  $D\to(T\oplus T^*)M$ then gives us a section $(\rho_D, \alpha_D) \in \Gamma(D^* \otimes(T\oplus T^*)M)$. Let us set $\overline{\rho}_D := \iota_{\rho_D}g \in \Gamma(D^* \otimes T^*M)$. 

In \cite{AHP} it was shown that  a two-dimensional sigma model  with the above data on the target space can be \emph{gauged} with respect to $D$, if $D$ can be equipped with two connections $\nabla^\pm := \nabla \pm \phi$, $\phi \in \Omega^1(\mathrm{End}(D))$, such that $(M,g,H,D)$ or $(E,V,D)$ satisfy the following compatibility conditions: 
\beqa \mathrm{Sym} \left(\overline{\nabla} \overline{\rho}_D - \phi^*(\alpha_D) \right)&=& 0 \, , \label{1}\\
\mathrm{Alt}\left(\overline{\nabla} \alpha_D - \phi^*(\overline{\rho}_D)-\tfrac{1}{2}\iota_{\rho_D}H \right)&=& 0 \, ,
\label{2}
\eeqa
where $\overline{\nabla}$ is the extension of $\nabla$ to $T^*M$ by means of the Levi-Civita connection of $g$, $\mathrm{Sym}$ and $\mathrm{Alt}$ denote the symmetrization and antisymmetrization projections in $T^*M \otimes T^*M$, respectively, and  $\phi^* \in \Omega^1(\mathrm{End}(D^*))$ is the 1-form valued map dual to $\phi$. 

In this case the variational problem can be gauged by extending 
the fields from maps $X \colon \Sigma \to M$ to vector bundle morphisms $a \colon T \Sigma \to D$---thus adding gauge field 1-forms $A \in \Omega^1(\Sigma,X^*D)$. Using the canonical splitting given by the above data, the independent field $A$ gives rise to its projections $A_{TM} \in \Omega^1(\Sigma,X^*TM)$ and $A_{T^*\!M}  \in \Omega^1(\Sigma,X^*T^*\!M)$  to $TM$ and $T^*\!M$, respectively. The gauged variational problem is then described symbolically by \cite{AHP,DSM}
\be S[a] = \int_\Sigma \tfrac{1}{2} \vert \vert \mathrm{d} X - A_{TM} \vert \vert^2  + \langle A_{T^*\!M} \stackrel{\wedge}{,}\mathrm{d} X -  \tfrac{1}{2}A_{TM}\rangle + \int H \, , \label{S}
\ee 
where for every $\nu \in \Omega^1(\Sigma,X^*TM)$ one has $\vert \vert \nu \vert \vert^2 \equiv (X^*\!g)(\nu \stackrel{\wedge}{,}  \ast \nu)$ with $\ast$ denoting the Hodge dual associated to $\gamma$---symbolically, since $S$ is not really a functional due to the Wess-Zumino term, while it still defines a unique variational problem for the field $a$ in the standard manner.

\begin{rem} While the definition of the variational problem of \eqref{S}---its Euler Lagrange equations \emph{and} its gauge equivalence of solutions---does not require the knowledge of  connections $\nabla^\pm$ satisfying the Equations \eqref{1} and  \eqref{2}, the off-shell gauge symmetries of a properly defined (possibly multi-valued) functional \eqref{S} do \cite{AHP}. Thus this applies also to an eventual quantization of $S$.
\end{rem}


\begin{rem} In \cite{AHP} the gauging is described by gauge fields taking values in an almost Lie algebroid $L$, where $L$ maps into a possibly singular small Dirac structure ${\cal D}\subset \Gamma(E)$. Here we displayed the simplified situation with $L=D$ only. While the more general situation is more complicated to describe, it is evidently more flexible and may have advantages even when ${\cal D}=\Gamma(D)$: consider, for example, a metric $g$ on a maximally symmetric target manifold $M$. Then the rank of the isometry Lie algebra $\mathfrak{g}$ is $\frac{n(n+1)}{2}>n$ and  choosing for $L$ the corresponding action Lie algebroid, $L=M \times \mathfrak{g}$, provides a simpler description of the gauge theory than the small Dirac structure it maps to. 
\end{rem}


\vspace{2mm}
\noindent {\bf 6.}
 It is one of the main purposes of the current letter to reformulate the conditions \eqref{1} and \eqref{2} in terms which are intrinsic to Courant algebroids. This is essentially achieved by means of the following statement.  
 \begin{theorem} A triple $(E,V,D)$ is a Dirac-Riemannian foliation \emph{iff}  there exist connections $\nabla^\pm$ such that the compatibility conditions \eqref{1} and \eqref{2} are satisfied.
\end{theorem}
\begin{proof} We use the canonical splitting to identify $E$ with $(T\oplus T^*)M$. We then have
$$V_+:=V=\{(u,\iota_u g):=u_+ \mid u\in TM\},\quad V_-:=V^\perp=\{(u,-\iota_u g):=u_- \mid u\in TM\}$$
where we implicitly defined maps $u \mapsto u_\pm$ from $TM$ to $V_\pm$. Consider in addition the  vector bundle maps $\pi_\pm \colon D \to T^*M , (X,\alpha) \mapsto \alpha \pm \iota_X g$. These are isomorphisms between $D_\pm := \pi_\pm(D) \subset T^*M$ and $D$, since $\iota_X (\alpha \pm \iota_X g) = \pm \vert \vert X\vert \vert^2$ vanishes only for $X=0$. 
We have $$ D^\perp\cap V_\pm= \{ u_\pm \mid u \in \mathrm{Ann}(D_\pm)\}.$$

As $V_D^{\hphantom{D}\perp}=D\oplus( D^\perp\cap V_-)$ and $[\Gamma(D),\Gamma(D)]\subset\Gamma(D)$, Condition \eqref{inv}  for $V_D \equiv D \oplus (D^\perp \cap V_+)$ can be restated as that  for every $(X,\alpha)\in \Gamma(D)$:
\begin{equation}\label{pf2}
\langle [(X,\alpha),u_+],v_-\rangle=0\quad\text{whenever $u_+$ and $v_-$ are in $D^\perp$},
\end{equation}
   i.e.\ whenever $u$ is annihilated by $D_+$ and $v$ by $D_-$
   On the other hand, one computes 
$$ [(X,\alpha),u_+] = ({\cal L}_{X}u, \iota_{({\cal L}_{X}u)}g +\iota_u ({\cal L}_{X}g - \mathrm{d}\alpha + \iota_{X}H)) \, . \label{bracket}
$$ 
giving
$$\langle [(X,\alpha),u_+],v_-\rangle=(\mathcal L_X g + i_X H - \mathrm{d}\alpha)(u,v).$$
This shows
\begin{lemma} \label{lemma} $V_D\equiv D \oplus (D^\perp \cap V)$ is a $D$-transverse generalized metric, \emph{iff} one has for all  $(X,\alpha)\in\Gamma(D)$
\be \label{pf1}\mathcal L_X g +i_X H-d\alpha\in\Gamma(D_+\otimes T^*M+T^*M\otimes D_-) \, .\ee \end{lemma}
Denote by $e_a := (X_a,\alpha_a)$ a local basis of $D$ and let $\beta_a^\pm = \pi_\pm(e_a)$ be the induced bases in $D_\pm$. On a local level, Condition \eqref{pf1}, and thus $V_D$ to be a transverse generalized metric, is  equivalent to the existence of locally defined coefficient 1-forms $({\omega}^\pm)_a^b$ such that\footnote{For the special case $H=\mathrm{d}B$ together with $D= \mathrm{graph}(-B)$, i.e.\ $\alpha=-\iota_XB$ in the description of $D$ above, the left-hand side of the following equation becomes simply $ {\cal L}_{X_a}(g+B)$, thus  reproducing  Equation (2.21) in \cite{JHEP}.}
$$  {\cal L}_{X_a}g + \iota_{X_a}H- \mathrm{d}\alpha_a   =  \beta_b^+ \otimes({\omega}^+)_a^b -(\omega^-)_a^b \otimes \beta_b^- \, .$$
This now is verified to be the local form of the equations \eqref{1} and \eqref{2}, with $({\omega}^\pm)_a^b$ being the connection coefficients of $\nabla^\pm$ in the chosen basis, $\nabla^\pm e_a = ({\omega}^\pm)_a^b \otimes e_b$. The global existence of the connections then follows by a standard argument using a partition of unity.
\end{proof} 
\begin{cor} The sigma model associated to  $(M,g,H)$ or $(E,V)$ on a pseudo-Riemannian 2-manifold $(\Sigma,\gamma)$ can be gauged along a (possibly singular) foliation ${\cal T}\subset \Gamma(TM)$, if there exists a small Dirac structure $D$ covering  ${\cal T}$, $\rho(\Gamma(D))={\cal T}$, which makes $(E,V,D)$  a Dirac-Riemannian foliation.
\end{cor}

\vspace{2mm}
\noindent {\bf 7.}  An exact Courant algebroid $E\to M$ gives rise to an infinite-dimensional symplectic manifold ${\cal M}$ which is the phase space of 2-dimensional sigma models: If a splitting of $E=(T\oplus T^*)M$ is chosen, giving rise to a closed 3-form $H\in\Omega^3(M)$, then we have ${\cal M}=T^*(LM)$ with the standard symplectic 2-form modified by the transgression of $H$. More naturally, ${\cal M}$ is the space of all vector bundle maps $f\colon TS^1\to E$, covering the  base map $f_0\colon S^1\to M$, such that $\rho\circ f\colon TS\to TM$ agrees with the tangent map of $f_0$.

A generalized metric $V\subset E$ then defines a function $\mathcal H_V$ on ${\cal M}$ which is the Hamiltonian of the corresponding two-dimensional sigma model:
$$\mathcal H_V(f)=\frac{1}{2}\int_{S^1}\langle f(\partial_\sigma),R_V f(\partial_\sigma)\rangle\, d\sigma \, ,$$
where $\sigma$ is the coordinate on $S^1$ and $R_V \colon E\to E$ is the reflection with respect to $V$.

A small Dirac structure $D\subset E$ defines the Lie algebra $\mathfrak g_D:=C^\infty(S^1)\otimes\Gamma(D)$ together with a Lie algebra map $\mu^* \colon \mathfrak g_D\to C^\infty({\cal M})$ given by  \cite{AS}
$$(\mu^*(s))(f)=\int_{S^1}\langle s,f\rangle.$$
The reduced phase space ${\cal M}/\!/\mathfrak g_D$ is composed of maps $f\colon TS^1\to D^\perp$, i.e.\ $f$'s in ${\cal M}$ satisfying $(\mu^*(s))(f)=0$ for all $s\in\mathfrak g_D$, modulo the action of $\mathfrak g_D$. A $D$-transverse generalized metric $W\subset E$ is then precisely what is needed to provide a Hamiltonian on the reduced phase space ${\cal M}/\!/\mathfrak g_D$. Define
$$\mathcal H_W([f])=\frac{1}{2}\int_{S^1}\langle P f(\partial_\sigma),R_W P f(\partial_\sigma)\rangle\,d\sigma \, ,$$
where $P$ denotes the natural projection from $D^\perp$ to $D^\perp/D$, and $R_W \colon D^\perp/D\to D^\perp/D$ is the orthogonal reflection with respect to $W/D$ which corresponds to the transverse generalized metric $W$.

\section*{Acknowledgement}

The work was supported by  the grant MODFLAT of the European Research Council and the NCCR SwissMAP of the Swiss National Science Foundation.

T.S.\ is grateful for the hospitality and the wonderful working conditions in the group of Anton Alekseev during visits when this work was began.


\bibliographystyle{utphys}

\end{document}